\documentclass[a4paper]{amsart}

\usepackage[utf8]{inputenc}
\usepackage{  amsmath,amstext,amsthm,verbatim,amssymb, amsrefs}
\usepackage{MnSymbol}
\usepackage{fullpage}
\DeclareMathAlphabet{\mathbbm}{U}{bbm}{m}{n}
\usepackage{tikz}

\usetikzlibrary{matrix}
\usetikzlibrary{shapes}
\usetikzlibrary{arrows}
\usetikzlibrary{calc,3d}
\usetikzlibrary{decorations,decorations.pathmorphing, decorations.pathreplacing}
\usetikzlibrary{through}
\tikzset{ext/.style={circle, draw,inner sep=1pt},int/.style={circle,draw,fill,inner sep=1pt},nil/.style={inner sep=1pt}}
\tikzset{exte/.style={circle, draw,inner sep=3pt},inte/.style={circle,draw,fill,inner sep=3pt}}
\tikzset{diagram/.style={matrix of math nodes, row sep=3em, column sep=2.5em, text height=1.5ex, text depth=0.25ex}}
\tikzset{diagram2/.style={matrix of math nodes, row sep=0.5em, column sep=0.5em, text height=1.5ex, text depth=0.25ex}}

\usepackage{tikz-cd}

\theoremstyle{plain}
  \newtheorem{thm}{Theorem}[section]
   
  \newtheorem{defi}[thm]{Definition}
  \newtheorem{prop}[thm]{Proposition}
  
  \newtheorem{cor}[thm]{Corollary}
  \newtheorem{lemma}[thm]{Lemma}
   
\theoremstyle{definition}
  
  \newtheorem{rem}[thm]{Remark}

\newcommand{\R}{{\mathbb{R}}}

\newcommand{\K}{{\mathbb{K}}}
\newcommand{\Q}{{\mathbb{Q}}}

\newcommand{\op}{\mathcal}

\newcommand{\SC}{\mathsf{SC}}

\newcommand{\ESC}{\mathsf{ESC}}

\newcommand{\Com}{\mathsf{Com}}

\newcommand{\FM}{\mathsf{FM}}

\newcommand{\bpm}{\begin{pmatrix}}
\newcommand{\epm}{\end{pmatrix}}

\newcommand{\e}{\mathsf{e}}

\newcommand{\La}{\Lambda}

\newcommand{\lD}{\mathsf{D}}

\newcommand{\dESC}{\mathsf{ESC}^c}

\newcommand{\bo}{{\mathbbm{1}}}

\newcommand{\lDM}{{\mathsf{M}}}

\newcommand{\mT}{{\mathcal{T}}}
\newcommand{\mC}{{\mathcal{C}}}

\newcommand{\id}{\mathit{id}}

\newcommand{\aor}{\rcirclearrowleft}
\newcommand{\aol}{\lcirclearrowright}
\newcommand{\beq}[1]{\begin{equation}
\label{#1}
}
\newcommand{\eeq}{\end{equation}}

\begin{document}
\title{ (Non-)formality of the extended Swiss Cheese operads }


%
\author{Thomas Willwacher}
\address{Department of Mathematics \\ ETH Zurich \\  
R\"amistrasse 101 \\
8092 Zurich, Switzerland}
\email{thomas.willwacher@math.ethz.ch}
%

\thanks{The author acknowledges partial support by the Swiss National Science Foundation (grant 200021\_150012 and the SwissMap NCCR). This work has been partially funded by the European Research Council, ERC StG 678156--GRAPHCPX}


\begin{abstract}
We study two colored operads of configurations of little $n$-disks in a unit $n$-disk, with the centers of the small disks of one color restricted to an $m$-plane, $m<n$.
We compute the rational homotopy type of these \emph{extended Swiss Cheese operads} and show how they are connected to the rational homotopy types of the inclusion maps from the little $m$-disks to the little $n$-disks operad.
\end{abstract}

\maketitle

\section{Introduction}

The little $n$-disks operad $\lD_n$ is a topological operad consisting of configurations of $n$-dimensional "little" disks inside the unit disk. It plays an important role in a variety of areas of mathematics.
In this paper we consider a two-colored extension of this operad, the extended Swiss Cheese operad $\ESC_{m,n}$.
The operations of output color one of this operad are identified with $\lD_n$. The operations in output color two are again configurations of non-intersecting $n$-dimensional disks in the unit $n$-disk, but now one distinguishes two kinds of disks, of input colors one and two: the disks of input color one can be freely placed anywhere in the disk, while the centers of the disks of input color two are restricted to lie on a fixed $m$-dimensional plane in the unit disk as shown below.
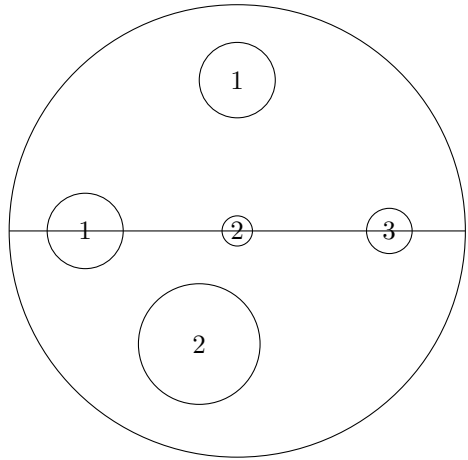
\[
\begin{tikzpicture}
\draw (0,0) circle (3);
\draw (-3,0) -- (3,0);
\draw (-2,0) circle (.5) (0,0) circle (.2) (2,0) circle (.3);
\draw (0,2) circle (.5) (-.5,-1.5) circle (.8) ;
\node at (-2,0) {1};
\node at (0,0) {2};
\node at (2,0) {3};
\node at (0,2) {1};
\node at (-0.5,-1.5) {2};
\end{tikzpicture}
\]
The operadic composition morphisms are defined by gluing one configuration of disks in place of one little disk of another configuration, just as for the little disks operad.
The goal of this paper is to study the rational homotopy type of the operads $\ESC_{m,n}$. Our main result will be the following.

\begin{thm}\label{thm:main}
The operads $\ESC_{m,n}$ are rationally formal for $n>m-1$ and $m\geq 1$. The corresponding operads of real chains $C(\ESC_{m,n})$ are not formal for $n=m-1$ and $m\geq 1$.
\end{thm}

By the first statement we mean that the (homotopy) cooperad of Sullivan's differential forms $\Omega_{PL}(\ESC_{m,n})$ can be connected to its cohomology by a zigzag of quasi-isomorphisms.

We will show the statements of Theorem \ref{thm:main} by linking the operads $\ESC_{m,n}$ to the rational homotopy theory of maps $\lD_m\to \lD_n$, and in particular derive the above result from the (non-)formality of the maps $\lD_m\to \lD_n$ shown in \cite{FW} and \cite{TW}.

\medskip

\subsection*{Remark about history and nomenclature.}
Although the definition is quite natural, we are not aware of appearances of the extended Swiss Cheese operads $\ESC_{m,n}$ in the literature. However, at least in codimension $n-m=1$ they have been invented by V. Turchin (personal communication), in connection with the study of the spaces of framed long non-$k$-overlapping immersions of $\R^m$ into $\R^n$, cf. \cite{DT}.
Concretely, there is an action of $\ESC_{m (m+1)}$ on the pair consisting of framed long embeddings and such immersions, akin to the action of the little disks operads on long embeddings as described by Budney \cite{Budney}.
The name "extended Swiss Cheese operad" is also due to Turchin.

\subsection*{Acknowledgements}
The author is very grateful for helpful and encouraging discussions with B. Fresse, A. Khoroshkin and V. Turchin.

\section{Notation and preliminaries}\label{sec:notation}
\subsection{Basic notions}
We generally work over a field $\K$ of characteristic zero, and more concretely either over $\K=\Q$ for the formality part of Theorem \ref{thm:main} and over $\K=\R$ for the non-formality part.

We use cohomological conventions, i.e., a differential graded (dg) vector space for us is a $\K$-vector space with a potentially unbounded grading, with differential of degree $+1$.
We abbreviate the notion of dg commutative algebra by dgca.
Note that our dgcas need not be concentrated in positive degrees unless otherwise stated.

\subsection{$\La$-operads and cooperads}
For an introduction to (co)operads we refer the reader to the textbook \cite{lodayval}, whose conventions we mostly follow.

We also recall the notion of $\La$-operad, cf. \cite{Frbook}.
Concretely, a $\La$-operad in topological spaces $\op T$ is a topological operad without operations in arity 0, but with maps $\iota_j: \op T(r+1)\to \op T(r)$, $j=1,\dots,r$.
These maps are required to satisfy axioms such that there is a natural operad structure on the spaces
\[
\op T_*(r) =
\begin{cases}
\op T(r) & \text{for $r\geq 1$} \\
* & \text{for $r=0$},
\end{cases}
\]
with the operadic insertion of the arity zero operation in slot $j$ being defined to be $\iota_j$.
In other word, the notion of $\La$ operad is merely an alternative formalization of an operad whose space of operations of arity zero is a point.

We call a cooperad in the category of dgcas a Hopf cooperad, in agreement with the literature.
A dg $\La$-Hopf cooperad is a cooperad in the category of dgcas, together with a (co-)$\La$ structure, i.e., maps $\op C(r)\to \op C(r+1)$ satisfying axioms dual to those for the maps $\iota_j$ above.
We refer the reader to \cite{Frbook} or the introductory sections of \cite{FW} for more details.

\subsection{Rational homotopy theory for operads}
For a topological space $X$ we shall denote the Sullivan differential graded commutative algebra (dgca) of polynomial differential forms (with rational coefficients) by
\[
\Omega_{PL}(X).
\] 
For good (e.g., simply connected, finite type) spaces $X$, we may recover the rational homotopy type of $X$ from the quasi-isomorphism type of $\Omega_{PA}(X)$, cf. \cite[Theorem 7.3.5]{Frbook}.

Now consider a topological operad $\op T$. Unfortunately, since the functor $\Omega_{PL}$ is lax monoidal, but not oplax monoidal, the collection of dgcas $\Omega_{PL}(\op T)$ does not immediately form a cooperad in the category of dgcas.
There are essentially two approaches to overcome this difficulty.
The first, due to Fresse \cite{Frbook}, \cite{Frextended} is to replace the functor $\Omega_{PL}(-)$ by a(n essentially quasi-isomorphic) functor $\Omega_\sharp(-)$ which sends (reduced) topological operads to cooperads in dgcas.
The second approach is to relax the notion of cooperad to homotopy cooperad.
Both approaches have advantages and disadvantages. Since for the present paper the precise framework one works in is not too relevant, let us briefly recall both approaches.
We will comment later how our constructions can be conducted in each framework.

\subsubsection{Fresse's rational homotopy theory for operads}\label{sec:rhtFresse}
\newcommand{\sSetOp}{\mathrm{sSet}\La\mathrm{Op}}
\newcommand{\dgsHopfcoOp}{\mathrm{dg}^*\La\mathrm{HopfOp}^c}
Following \cite{Frbook} we consider the category $\sSetOp$ of reduced $\La$-operads in simplicial sets, and the category $\dgsHopfcoOp$ of reduced, non-negatively graded dg $\La$-Hopf cooperads.
The Sullivan realization functor $G_\bullet$ from dg commutative algebras to simplicial sets is symmetric monoidal and hence induces a functor 
\[
G_\bullet: \dgsHopfcoOp \to \sSetOp^{op} .
\]
One then defines the functor 
\[
\dgsHopfcoOp \leftarrow \sSetOp^{op} : \Omega_\sharp
\]
to be the right adjoint to $G_\bullet$, see \cite[Proposition II.12.1.2]{Frbook} for the explicit construction.
Arity-wise, the functor $\Omega_\sharp$ is quasi-isomorphic to $\Omega_{PL}$ under mild conditions (cd \cite[Theorem 2.3]{Frextended}) and shall serve as our operadic upgrade of the latter functor.

Furthermore, one may define model structures on the categories of $\sSetOp$ and $\dgsHopfcoOp$ such that the functors $G_\bullet$ and $\Omega_\#$ form a Quillen adjunction.
We refer the reader to \cite{Frbook} or to the concise recollection of the introductory sections of \cite{FW} for the precise definitions.

For one later application we just mention the following property.
\begin{lemma}[B. Fresse, private communication]\label{lem:cofibrant}
Suppose that $\op C$ is a cofibrant dg $\La$ Hopf cooperad. Then the morphisms 
\[
\op C(r)\to\op C(r+1) 
\]
which define the $\La$ structure are cofibrations of dg commutative algebras.
\end{lemma}
\begin{proof}[Proof sketch]
We may check the statement in the category of dg $\La$ Hopf collections, forgetting the cooperad structure and arity one operations, as the forgetful functor is part of a Quillen adjunction.
The model structure on the category of dg $\La$ Hopf collections is cofibrantly generated and transported from the category of coaugmented dg collections, cf. \cite[Figure 1]{FW} and explanations thereafter.
In particular, the generating cofibrations are "free" morphisms of dg $\La$ Hopf collections.

By standard results for cofibrantly generated model categories \cite[Proposition 2.1.18]{Hovey} any cofibration in such a category can be obtained as a retract of a transfinite composition of pushouts along generating cofibrations.
The desired property in the Lemma is stable under retracts, as is easily checked. 
It is also preserved by taking pushouts under generating cofibrations, and by directed colimits, so we conclude the desired result.
(Every pushout here adds new "free generators" with respect to the $\La$ Hopf structure, with prescribed differential.)
\end{proof}

We finally remark that the definition of model structures and the functor $\Omega_\#$ may also be extended to categories of colored operads.

\subsubsection{Homotopy (co)operads}
We will use the notion of homotopy cooperad proposed by Lambrechts and Voli\'c \cite{LV}, and spelled out in somewhat more detail in \cite{KW}.

To briefly recall the definition, one defines a symmetric monoidal category of forests $\mT$.
The monoidal product is the disjoint union. The morphism are generated by (i) isomorphisms of trees (ii) contraction of vertices and (iii) cutting of edges.
A (non-unital) homotopy operad in the symmetric monoidal category $\mC$ is then a symmetric monoidal functor  
\[
\mT \to \mC.
\]
For example, any operad $\op P$ in $\mC$ gives rise to a homotopy operad by sending the tree $T$ to the tree-wise tensor product
\[
\bigotimes_T \op P,
\]
and using the operadic composition to define the images of the contraction morphisms.

Dually, a (non-counital) homotopy cooperad in $\mC$ is a functor
\[
\mT \to \mC^{op}.
\]
The advantage is that with this definition the lax monoidal functor $\Omega_{PL}(-)$ gives rise to a functor from the category of operads in topological spaces to the category of homotopy cooperads in dgcas. Abusing notation we shall denote this functor by $\Omega_{PL}$ as well.
Concretely, for $\op T$ a topological operad the homotopy cooperad $\Omega_{PL}(\op T)$ sends the tree $T$ to the dgca
\[
\Omega_{PL}(\times_T \op T).
\]

The definitions can be extended easily to the case of colored operads. One simply has to alter the category $\op T$ to consist of colored forests, all of whose edges are colored by one color in the given set of colors.

\subsection{The little $n$-disks operad}\label{sec:little disks}
We shall recall here some well-known facts about the little $n$-disks operads $\lD_n$, referring to \cite{Sinha} for detailed derivations and pointers to the original literature.
We denote the rational homology operad of $\lD_n$ by 
\[
\e_n := H_\bullet(\lD_n)
\]
The cohomology $\e_n^*$ of the little disks operads was computed by Arnold \cite{Arnold} for $n=2$ and by F. Cohen \cite{Co} for higher $n$.
One has the presentation as a dgca
\[
\e_n^*(r) 
=\Q[\omega_{ij}\mid 1\leq i\neq j\leq r] /\sim
\]
where $\omega_{ij}$ is of degree $n-1$ and the relations are 
\begin{align*}
\omega_{ij}= (-1)^n\omega_{ji} \\
\omega_{ij}\omega_{jk}+\omega_{jk}\omega_{ki}+\omega_{ki}\omega_{ij}=0.
\end{align*}

The cooperads $\e_n^*$ are naturally $\La$ cooperads in dgcas, the $\La$-structure being induced by forgetting disks from a configuration of disks. We shall need the following result below.

\begin{lemma}\label{lem:freeness}
The module $\e_n^*(r+k)$ over the dgca $\e_n^*(r)\subset \e_n^*(r+k)$ is free for all $n\geq 1$, $r\geq 0$ and $k\geq 0$.
\end{lemma}
\begin{proof}
 A basis of $\e_n(s)$ is indexed by pairs $(I,f)$, where $I\subset \{1,\dots,s\}$ and $f:I\to \{1,\dots, s\}$ is an increasing function, i.e., $f(i)>i$.
The corresponding monoial is 
\[
 \omega_{I,f} = \prod_{i\in I} \omega_{if(i)}.
\]
(We refer the reader to \cite{Sinha} for this statement and further details.)
For $s=r+k$ we think of $\e_n^*(r)$ as generated by the $\omega_{ij}$, with $i,j\geq k+1$.
Then it is clear that a basis of $\e_n^*(r+k)$ as module over $\e_n^*(r)$ is given by the monomials $\omega_{I,f}$ as above with $I\subset \{1,\dots,k\}$.
\end{proof}

\section{A colored operad associated to a $\La$ operad map}\label{sec:ESC construction}
The goal of this section is to construct a colored topological operad $\ESC_f$ from a map of topological $\La$ operads
\[
f \colon \op P \to \op Q.
\]

The definition is made so that if we take for $f$ the map $\lD_m\to \lD_n$, then $\ESC_f\simeq \ESC_{m,n}$.
The dual construction will produce a two colored dg Hopf cooperad $\dESC_F$ from a map of dg Hopf $\La$ cooperads
\[
F \colon \op C \to \op D.
\]
Applying this construction to a rational dg Hopf cooperad model for the map $\lD_m\to \lD_n$ we will obtain rational (dgca) models for $\ESC_{m,n}$.

\subsection{The trivial colored operad associated to an operad}
Given an ordinary operad $\op P$ we may "trivially" build a two colored operad $\op P_{triv}$ by declaring that
\begin{align*}
\op P_{triv}^1(k,l) &= 
\begin{cases}
* & \text{for $l=0$} \\
\emptyset & \text{for $l>0$} 
\end{cases}\\
\op P_{triv}^2(k,l) &= \op P(l) 
\end{align*}
and with the obvious compositions derived from the compositions in $\op P$.
Clearly, the construction $\op P \to \op P_{triv}$ is functorial,

\subsection{The painted colored operad associated to an operad}
Given an ordinary operad $\op P$ we may build another two colored operad $\op P_{paint}$ by declaring that
\begin{align*}
\op P_{paint}^1(k,l) &= 
\begin{cases}
\op P(k) & \text{for $l=0$} \\
\emptyset & \text{for $l>0$} 
\end{cases}\\
\op P_{paint}^2(k,l) &= \op P(k+l) 
\end{align*}
and with the obvious compositions derived from the compositions in $\op P$. In other words, the operad $\op P_{paint}$ is defined by declaring (or "painting") some inputs and the output in color 2.
Mind that we require here that if any inputs are "painted" in color 2, then so must be the output.
Again, the construction $\op P \to \op P_{paint}$ is obviously functorial,

\subsection{A map between these colored operads}
Now suppose that $\op P$ in addition carries a $\La$ structure. Then we define a map of two colored operads $\op P_{paint}\to \op P_{triv}$ such that the map 
\begin{align*}
 \op P(k+l) = \op P_{paint}^2(k,l)  \to  \op P_{triv}^2(k,l) = \op P(l) 
\end{align*}
is the $\La$ map associated to the inclusion $[l]\to [k+l]$. In other words, one formally "fills the first $k$ inputs with the unit".
In the other arity components the map $\op P_{paint}\to \op P_{triv}$ is necessarily trivial, since so are the corresponding components of $\op P_{triv}$.

It is not hard to check that the above map of two colored collections $\op P_{paint}\to \op P_{triv}$ is indeed a map of colored operads.
Furthermore, the construction is evidently functorial in $\op P$.

\subsection{The colored operad $\ESC_f$}\label{sec:ESCfdef}
Now suppose we have a map of $\La$ operads 
\[
f\colon \op P\to \op Q.
\]
Using the maps above and the functoriality we can then build a zigzag $\op P_{triv} \to \op Q_{triv} \leftarrow \op Q_{paint}$. Our main definition (and the target of this section) is then the following.
\begin{defi}
We define the two colored operad $\ESC_f$ associated to the map of $\La$ operads $
f\colon \op P\to \op Q$ to be the homotopy pullback
\begin{equation}\label{equ:ESCpullback}
\begin{tikzcd}
\ESC_f \ar{r} \ar{d} & \op Q_{paint} \ar{d}\\
\op P_{triv} \ar{r} & \op Q_{triv}
\end{tikzcd}
.
\end{equation}
\end{defi}

In the case that all $\La$ maps of $\op Q$ are fibrations, as is the case in all examples of relevance in this paper, we will quietly replace the homotopy pullback by an ordinary pullback.

\begin{rem}\label{rem:weakequiv}
Note that the construction $f\mapsto \ESC_f$ is functorial. Moreover, weak equivalence between morphisms $f:\op P\to \op Q$ and $f':\op P'\to \op Q'$, i.e., a commutative diagram of the form
\[
\begin{tikzcd}
\op P \ar{r}{f} \ar{d}{\simeq}& \op Q \ar{d}{\simeq}\\
\op P'  \ar{r}{f'} & \op Q'
\end{tikzcd},
\]	
induces a weak equivalence between two colored operads $\ESC_f\simeq \ESC_{f'}$ .
\end{rem}

\begin{rem}\label{rem:ESCfESCmn}
The components of the Extended Swiss Cheese operads fit into pullback squares 
\[
\begin{tikzcd}
\ESC_{m,n}(k,l) \ar{r} \ar{d} & \op \lD_n(k+l) \ar{d}\\
\lD_m(l) \ar{r}{f} & \lD_n(l)
\end{tikzcd}
\]
with the right-hand maps being fibrations. Hence $\ESC_{m,n}$ is the homotopy pullback in the diagram \eqref{equ:ESCpullback}.
Inspection then shows that indeed the operad structure inherited from being a pullback of operads \eqref{equ:ESCpullback} coincides with the ``naturally defined'' one through gluing of discs, and hence $\ESC_f\cong \ESC_{m,n}$ as desired.
\end{rem}

\subsection{Fulton-MacPherson-Axelrod-Singer variant}
For some application we want our operads to be reduced, i.e., being a point in arity one, and the little disks operads $\lD_n$ obviously are not.
However, we may replace them by the Fulton-MacPherson-Axelrod-Singer compactification of the configuration spaces of points $\FM_n$, see \cite{GJ}. They satisfy in particular that $\FM_n(1)=*$.
The inclusion maps $f:\lD_m\to \lD_n$ can be replaced by similar inclusion maps 
\[
g:\FM_m\to \FM_n,
\]
that are induced from the standard embedding $\R^m\subset \R^n$.
We may hence define the Fulton-MacPherson-Axelrod-Singer variant of the extended Swiss Cheese operad to be 
\[
\FM_{m,n} := \ESC_g.
\]

\subsection{The dual construction}

Next suppose we have a map of dg Hopf $\La$ cooperads
\[
F \colon \op C\to \op D.
\]
We then define the colored dg Hopf homotopy cooperad $\dESC_F$ by (the dual of) our construction above, i.e., as the homotopy pushout
\[
\begin{tikzcd}
 \op C_{triv} \ar{r} \ar{d} & \op \op C_{paint} \ar{d}\\
\op D_{triv} \ar{r} &\dESC_{F}\, .
\end{tikzcd}
\]

More concretely, the left and upper arrows gives rise to a homotopy cooperad in the category of diagrams of dg commutative algebras of shape $\bullet \leftarrow \bullet \rightarrow \bullet$.
The homotopy pushout (derived tensor product) is then a symmetric lax monoidal functor from this category to the category of dgcas, and hence gives rise to a homotopy dg Hopf cooperad.
In particular, the homotopy push-out agrees with the "arity-wise" homotopy push-out in the underlying category of dg commutative algebras.
For trees that are individual corollas we have push-out diagrams of dgcas
\[
\begin{tikzcd}
  \op C(l)  \ar{r} \ar{d} & \op  \op C(k+l) \ar{d}\\
 \op D(l) \ar{r} &\dESC_{F}(k,l).
\end{tikzcd}
\]

\begin{rem}
In general, for the homotopy pushout of the diagram of dgcas $B\leftarrow A \to C$ we may pick here the concrete "bar" realization 
\[
\bigoplus_{n\geq 0} B\otimes (A[1])^{\otimes n} \otimes C.
\]
Note however that this has two disadvantages: (i) the functor is lax monoidal and not oplax monoidal, hence in our context we obtain a homotopy cooperad not a cooperad and (ii) the complex is not concentrated in non-negative degrees, even if $A,B,C$ are. Hence we can define $\dESC_F$ (only) as a colored homotopy dg Hopf cooperad, with the underlying category being unbounded cochain complexes.

If we restrict this construction to non-negatively graded dg $\La$-Hopf cooperads such that all the $\La$-maps (in the source of $F$) are cofibrations of dgcas, then the homotopy pushout can be replaced by an ordinary pushout.
In particular, the construction $\dESC_F$ yields an honest (i.e., not ``homotopy'') non-negatively graded colored dg Hopf cooperad.
\end{rem}

\section{The colored operad $\ESC_{m,n}$}
In this section we restrict to the special case of the operad $\ESC_{m,n}$ and give a proof of our main Theorem \ref{thm:main}.

\subsection{Cohomology}\label{sec:cohomology}
First let us compute the cohomology cooperad $H(\ESC_{m,n}):= H^\bullet(\ESC_{m,n}, \Q)$.

For $n>m\geq 1$ we denote the canonical map from $\lD_m$ to $\lD_n$ by 
\[
f: \lD_m\to \lD_n.
\]
We note that the induced map on cohomology
\[
H(f)\colon H( \lD_n)\to H(\lD_m)
\]
factorises through the cocommutative cooperad $\Com^*$, i.e., it sends all generators $\omega_{ij}$ (see section \ref{sec:little disks}) to zero.

Furthermore, note that since $H(\e_n(r+s))$ is free as a $H(\e_n(r))$-module by Lemma \ref{lem:freeness}, we can replace the homotopy pushouts in the construction of $\dESC_{H(f)}$ by ordinary pushouts and understand $\dESC_{H(f)}$ as a non-negatively graded colored Hopf cooperad (with zero differential).

Note also that from this we see that there is a natural map
\beq{equ:Hf}
\dESC_{H(f)} \to H(\ESC_{m,n}).
\eeq
This in turn implies that the Eilenberg-Moore spectral sequence associated to the pullback diagrams collapses at the $E^2$ page.
Provided the base spaces in all pullback diagrams are simply connected, i.e., provided that $n\geq 3$, we immediately conclude from the standard convergence result for the Eilenberg-Moore spectral sequence that \eqref{equ:Hf} is an isomorphism.
In the remaining case of $n=2$ and $m=1$ we do not know a one-line proof.
However, one can easily repeat the arguments leading to the computation of $H(\lD_n)$, see for example \cite{Sinha}.
Essentially equivalently one can invoke the refined convergence result for the Eilenberg-Moore spectral sequence of Dwyer \cite{Dwyer74}, applied to the fibrations 
\[
\lD_n(r+s)\to \lD_n(r+s-1)\to \cdots  \to \lD_n(r).
\]

We leave the detailed argument in this special case ($n=m+1=2$) to the reader and merely state the result.
\begin{prop}\label{prop:cohomology}
The natural map of colored dg Hopf cooperads
\[
\dESC_{H(f)} \to H(\ESC_{m,n})
\]
is an isomorphism for all $n>m\geq 1$.
\end{prop}

\subsection{Proof of the first part of Theorem \ref{thm:main}}\label{sec:firstpartofproof}
In this section we want to show the rational formality of $\ESC_{m,n}$ for $n-m\geq 2$ and $m\geq 1$. To show this we want to show that a rational model (i.e., a colored homotopy dg Hopf cooperad) for $\ESC_{m,n}$ is formal.
Let us first build such a model. 
We again denote the canonical map from $\lD_m$ to $\lD_n$ by 
\[
f: \lD_m\to \lD_n.
\]
Pick some rational models (dg Hopf $\La$-cooperads) $\lDM_m$ and $\lDM_n$ for $\lD_m$ and $\lD_n$, together with a rational model for the above map
\[
F : \lDM_n\to \lDM_m.
\] 

Now we would like to conclude that $\dESC_{F}$ is indeed a rational model for $\ESC_f$, i.e., quasi-isomorphic to $\Omega_{PL}(\ESC_f)$. For $m\geq 2$ this can be deduced from the statement that the model of the pullback is the pushout of the models, see Proposition 15.8 of \cite{FHT} or Theorem 2.4 of \cite{HessRHT}.
However, for our purposes the conditions in the aforementioned results from the literature are too restrictive to handle the cases $m=1,2$. Fortunately, however, we may simply check the desired statement by hand since all cohomologies can be computed.
To this end it suffices to take $\lDM_m=\Omega_{PL}(\lD_m)$ and $\lDM_n=\Omega_{PL}(\lD_n)$ with $F=f^*$.
Then merely by functoriality and properties of the (co)limit we have a map
\[
\dESC_{F} \to \Omega_{PL}(\ESC_{m,n}).
\]
We want to show that this map is a quasi-isomorphism.
To check this we have to check that the natural maps
\beq{equ:tobesame}
\Omega_{PL}(\lD_n(k+l)) \hat \otimes_{\Omega_{PL}(\lD_n(l))} \Omega_{PL}(\lD_m(l)) 
\to
\Omega_{PL}(\ESC_{m,n}(k,l))
\eeq
are quasi-isomorphisms. But this is easily checked by hand: The cohomology of the right-hand side has been computed in section \ref{sec:cohomology}.
On the other hand, the left-hand side is a second quadrant double complex and there is a convergent spectral sequence such that
\[
E^2 = H(\lD_n(k+l)) \hat \otimes_{H(\lD_n(l))} H(\lD_m(l)).
\] 
But this expression is equal to $H(\ESC_{m,n}(k,l))$ by Proposition \ref{prop:cohomology}, and the cohomology is represented by cocycles in the original complex, so the spectral sequence abuts here, and the left-hand and right-hand side of \eqref{equ:tobesame} have isomorphic cohomology.
It remains to check that the map \eqref{equ:tobesame} induces an isomorphism, but this is clear looking at the explicit representatives.

Hence we find that the colored homotopy dg Hopf cooperad $\dESC_{F}$ is indeed a rational model for $\ESC_{m,n}$.

Now, by the relative intrinsic formality theorem of \cite{FW} and the assumption that $n-m\geq 2$ the map $F$ is formal, i.e., it can be connected to the cohomology map 
\[
H(F) : \e_n^* \stackrel{*}{\to} \e_m^* 
\]
by a zigzag of quasi-isomorphisms of dg $\La$-Hopf cooperads. By Remark \ref{rem:weakequiv} it then follows that the (much simpler) colored dg Hopf cooperad $\dESC_{H(F)}$ is connected by a zigzag of quasi-isomorphisms to $\dESC_{F}$, and is hence also a rational model for $\ESC_{m,n}$.
Furthermore, as in section \ref{sec:cohomology} $\dESC_{H(F)}$ can be taken to have zero differential and hence is trivially formal and  $\dESC_{H(F)}=H^*(\dESC_{m,n})$. This shows the first claim of Theorem \ref{thm:main}.

\begin{rem}
We note that here we can construct the above zigzag of quasi-isomorphisms connecting $\dESC_F$ and $\dESC_{H(F)}$ either in the category of two colored homotopy dg Hopf cooperads, or in the category of honest two colored dg Hopf cooperads concentrated in non-negative degrees.
For the latter statement we just need to replace the functor $\Omega_{PL}$ by the (two colored version of the) operadic upgrade $\Omega_\#$ as in section \ref{sec:rhtFresse}.
Furthermore we have to replace (resolve) all dg $\La$ Hopf cooperads in the zigzag connecting $F$ and $H(F)$ by cofibrant objects. Then we can use Lemma \ref{lem:cofibrant} to see that all homotopy pushouts in the constructions $\dESC_{(\cdots)}$ can be replaced by ordinary pushouts, so that all colored cooperads occurring in our zigzag are honest colored dg Hopf cooperads, concentrated in non-negative cohomological degrees.
\end{rem}

\section{Recovering the operad map}

Let $\op S$ be any two colored operad in dg vector spaces. From $\op S$ may extract two one colored operads
\begin{align*}
\op P &:=\op S^2(0,-)
&
\op Q &:=  \op S^1(-, 0) 
\end{align*}
and an operadic $\op P$-$\op Q$-bimodule 
\[
\op M := \op S^2(-,0).
\]
We call the operad $\op S$ to be of "Swiss Cheese type" if the $\op P$-$\op Q$-bimodule $\op M$ is a right torsor, cf. Definition \ref{def:torsor} in section \ref{app:torsors}. For example, the operads of chains of the ordinary and the extended Swiss Cheese operads are of Swiss Cheese type. Furthermore, the operads of chains of operads arising from the construction $\ESC_f$ of section \ref{sec:ESC construction} above are of Swiss Cheese type, if we require that the space of unary operations in the target of $f$ is contractible.

As shown in the Appendix, any operadic $\op P$-$\op Q$-torsor encodes an operad map up to homotopy 
\[
\op P \dashrightarrow \op Q. 
\]
Hence any two colored operad of Swiss Cheese type encodes such an operad map. Furthermore, a quasi-isomorphism $\op S \to \op S'$ between such two colored operads induces a quasi-isomorphism of the associated right torsors
\[
 \begin{tikzcd}[column sep=0.5em]				
 \op P \arrow{d}{\sim} &\aol & \op M \arrow{d}{\sim} & \aor & \op Q \arrow{d}{\sim} \\
 \op P' & \aol  & \op M' &\aor & \op Q'
 \end{tikzcd}\, .
\]
By Proposition \ref{prop:mapfunctorial} we hence see that quasi-isomorphic two colored operads of Swiss cheese type encode quasi-isomorphic operad maps between their components of colors 1 and 2.
In particular we find the following result.
\begin{prop}\label{prop:mapnonformal}
Two two-colored operads $\op S$ and $\op S'$ of Swiss Cheese type can be quasi-isomorphic only if the induced operad maps $\op S^2(0,-) \dashrightarrow \op S^1(-,0)$ and $ (\op S')^2(0,-) \dashrightarrow (\op S')^1(-,0)$ are quasi-isomorphic.
\end{prop}

\subsection{The proof of Theorem \ref{thm:main}}
Having shown the first part of Theorem \ref{thm:main} in section \ref{sec:firstpartofproof}, it suffices to show the second part, namely that the operad of chains on the extended Swiss Cheese operad $\FM_{m,n}$ is not formal in codimension $n-m=1$.
Indeed, if it were, then by Proposition \ref{prop:mapnonformal} the induced map of operads 
\[
C(\FM_{n-1}) \to C(\FM_n)
\]
was formal. However, it has been shown in \cite{TW} that the above operad map is \emph{not} formal, and hence the non-formality of $C(\FM_{n-1,n})$ and Theorem \ref{thm:main} follow.
\hfill\qed

\subsection{A new proof of non-formality of the standard Swiss Cheese operad} 
Note that the above arguments apply mutatis mutandis also to the operads of chains of the standard Swiss Cheese operads $\SC_n$ \cite{V}, which are also of Swiss Cheese type. 
This gives a second proof of the following result, shown by Livernet.
\begin{cor}[\cite{livernet}] 
The chains operads $C(\SC_n)$ of the Swiss Cheese operads are not formal.
\end{cor}

\begin{rem}
The above Corollary is relatively easy to show directly by standard obstruction theory, because the obstruction to the existence of a map from a cofibrant resolution of $H(\SC_n)$ to $C(\SC_n)$ appears already in arity 2, where the calculations are very manageable. 
On the other hand, the obstruction for the operad map $\lD_{n-1}\to\lD_{n}$ to be formal appears only in arity 4 \cite{TW}. 
In this arity there are already many operations in the resolution of $H(\ESC_{n-1,n})$, making the calculations quite intricate. That is why we refrain from using a direct obstruction theoretic computation to show the non-formality of $\ESC_{n-1,n}$ (-which should be possible-), but rather stick to the more conceptual methods of this section.
\end{rem}



%

\newcommand{\iHom}{\mathcal{H}\mathit{om}}
\newcommand{\OHom}{\mathcal{OH}\mathit{om}}
\newcommand{\OEnd}{\mathcal{OE}\mathit{nd}}

\newcommand{\Mod}{\mathsf{Mod}}
\newcommand{\Ind}{\mathrm{Ind}}
\renewcommand{\Bar}{B}
\newcommand{\Barc}{\Omega}

\section{Operadic bimodules and maps between operads}\label{app:torsors}

In this section we show that certain operadic bimodules can be used to encode operad maps.
This is used above to show the non-formality part of Theorem \ref{thm:main}.
We shall note that the idea of encoding operad maps by torsors has already been introduced in \cite{CW}. Here we present an extended version of those ideas and constructions.

\subsection{Recollections on operadic right modules}
Let $\op Q$ be a dg operad.
An operad is a monoid in the category of symmetric sequences with the "plethysm" monoidal product $\circ$, see \cite{lodayval}.
An operadic right $\op Q$-module is a right module for the monoid $\op Q$. Concretely, it is a
symmetric sequence $\op M$ together with a map 
\[
\op M\circ \op Q \to \op M
\]
satisfying some associativity and unit axioms. 
We shall need several facts about the category $\Mod_{\op Q}$ of operadic dg right $\op Q$-modules, which we mostly recall from \cite{Frbookmod}. First, $\Mod_{\op Q}$ comes equipped with a model structure obtained by transfer along the forgetful functor to dg symmetric sequences.
Concretely, the weak equivalences are the quasi-isomorphisms, the fibrations are the arity-wise surjective maps and the cofibrations are defined via the lifting property with respect to acyclic fibrations. In particular, all objects are fibrant.
Furthermore, since right $\op Q$-modules are "linear objects" $\Mod_{\op Q}$ is naturally a dg category. We denote by $ \iHom_{\op Q}(\op M,\op N)$ the dg vector space of morphisms between the right $\op Q$-modules $\op M$ and $\op N$.
(Actual module morphisms are then degree zero cocyles in the dg vector space.)
Finally, the following result is standard.
\begin{lemma}[Existence of homotopy inverses]\label{lem:htpy inverse}
Let $f: \op M  \to \op N$ be a weak equivalence between cofibrant objects in $\Mod_{\op Q}$.
Then there exists a weak equivalence $g: \op N  \to \op M$ and homotopies (right $\op Q$-module morphisms) $h_1:\op M\to \op M[-1]$ and $h_2:\op N\to \op N[-1]$ such that  
\begin{align*}
\id_{\op M} -gf&=dh_1+h_1d & 
 \id_{\op N} -fg&=dh_2+h_2d
\end{align*}
If $f$ is a fibration (resp. a cofibration) then we can take $h_2=0$ (resp. $h_1$=0). 
\end{lemma}
\begin{proof}
	First we obtain $g$ and $h_2$ satisfying the second equation by lifting the following diagram
	\[
	\begin{tikzcd}
	* \ar{r} \ar{d} & \op M \oplus (\op N \oplus \op N[-1], d) \ar{d}{(f,\id_{\op N},0)}\\
	\op N \ar{r}{=} 
	& \op N
	\end{tikzcd}\, .
	\]
	(Note that the right-hand map is surjective and hence a fibration.)
	Then we obtain similarly a map $\tilde f:\op M\to \op N$ and a homotopy $\tilde h:\op M\to \op M[-1]$ such that $\id-g\tilde f=d\tilde h + \tilde h d$.
	We then obtain the desired homotopy as
	\[
	h_1 := (\id-g f)\tilde h + gh\tilde f.
	\]
	Indeed, we check 
	\[
	dh_1+h_1d = (\id-g f)(\id-g \tilde f) + g(\id-fg)\tilde f = \id -gf.
	\]
	If $f$ is a fibration we can just omit the two $\op N$ in the upper right corner of the above square in the first step and set $h_2=0$.
	If $f$ is a cofibration we define $g$ by lifting the diagram
	\[
		\begin{tikzcd}
	\op M \ar{r}{=} \ar{d}{f} & \op M \ar{d}\\
	\op N \ar{r} 
	& *
	\end{tikzcd}
	\]
	instead and set $h_1=0$. 
\end{proof}

\subsection{Endomorphism operad of right $\op Q$-modules}
Next, there is a symmetric monoidal structure on $\Mod_{\op Q}$ such that for right modules 
$\op M, \op N\in \Mod_{\op Q}$ the tensor product is defined via (cf. \cite[section 2.1.7]{Frbookmod})
\[
(\op M \otimes \op N)(r) = \bigoplus_{r_1+r_2=r} \mathrm{Ind}_{S_{r_1}\times S_{r_2}}^{S_r} \op M(r_1)\otimes \op M(r_2).
\]

For right $\op Q$ modules $\op M$ and $\op N$ we define the symmetric sequence $\OHom_{\op Q}(\op M, \op N)$ such that 
\[
\OHom_{\op Q}(\op M, \op N)(r) := \iHom_{\op Q}(\op M^{\otimes r},\op N).
\]
In the special case $\op M=\op N$ this is naturally a dg operad, the endomorphism operad of $\op M$,
\[
\OEnd_{\op Q}(\op M) = \OHom_{\op Q}(\op M,\op M) =  \iHom_{\op Q}(\op M^{\otimes -},\op N).
\]
The endomorphism operad construction is not a functor in $\op M$. However, it still has some good functorial properties as the following result shows.
\begin{prop}\label{prop:functorial}
	To a weak equivalence $f:\op M\to \op N$ between cofibrant right $\op Q$-modules one can assign an operadic quasi-isomorphism up to homotopy (a zigzag) $F : \OEnd_{\op Q}(\op M) \dashrightarrow \OEnd_{\op Q}(\op N)$ in such a way that the following properties hold:
	\begin{enumerate}
		\item If $f$ is the identity map so is $F$. 
		\item The assignment $f\mapsto F$ is compatible with compositions in the following sense. If $g:\op N\to \op N'$ is another weak equivalence between cofibrant objects that is assigned a zigzag $G : \OEnd_{\op Q}(\op N) \dashrightarrow \OEnd_{\op Q}(\op N')$, then the composite $g\circ f$ is assigned a zigzag homotopic to the composition of $G$ and $F$.
		\item $F$ can be realized as a zigzag $\OEnd_{\op Q}(\op M)\xleftarrow{\sim} X \xrightarrow{\sim} \OEnd_{\op Q}(\op M)$ such that the diagram of symmetric sequences 
		\begin{equation}\label{equ:functorial pullback}
		\begin{tikzcd}
		X \ar{r} \ar{d} & \OEnd_{\op Q}(\op M) \ar{d}{ f\circ - } \\
		\OEnd_{\op Q}(\op M) \ar{r}{-\circ f} & \OHom_{\op Q}(\op M,\op N)
		\end{tikzcd}
		\end{equation}
		commutes.
	\end{enumerate}
	In fact, we shall see that these conditions determine $F$ uniquely up to homotopy.
\end{prop}

For the proof we shall need the following auxiliary result (cf. \cite[section II.2]{Frbook}).
\begin{lemma}
Let $f:\op M\to \op N$ be a weak equivalence between cofibrant right $\op Q$-modules, and let $\op N'$ be another cofibrant right $\op Q$-module. Then the maps of symmetric sequences
\begin{align*}
\OHom_{\op Q}(\op N, \op N') &\xrightarrow{-\circ f} \OHom_{\op Q}(\op M, \op N') \\
&\text{and} \\
\OHom_{\op Q}(\op N', \op M) &\xrightarrow{f\circ -} \OHom_{\op Q}(\op N' ,\op N)
\end{align*}
given by post- or precomposition with $f$ (or $f^{\otimes -}$) are weak equivalences.
If $f$ is a cofibration then the first map is a fibration (i.e., surjective) and if $f$ is a fibration then the second map is a fibration.
\end{lemma}
\begin{proof}[Proof sketch]
	We can use Lemma \ref{lem:htpy inverse} to obtain homotopy data	
	\[
	\begin{tikzcd}
	\ar[loop left]{}{h_1} \op M \ar[shift left]{r}{f} & \op N  \ar[shift left]{l}{g} \ar[loop right]{}{h_2}
	\end{tikzcd}\, .
	\]
	This then induces similar homotopy data for the above maps of hom spaces in each arity separately, showing that they are quasi-isomorphisms.
	In the case of $f$ being a (co)fibration we can take one of $h_2$ ($h_1$) to be zero, and hence obtain induced homotopy data with zero homotopy on one side between the hom spaces. This then shows that our map has a one sided inverse and in particular needs to be surjective in the cases stated.
\end{proof}

For $f:\op M\to \op N$ as above let us introduce the notation 
\[
\OEnd_{\op Q}(f) := \OEnd_{\op Q} (\op M)\times_{\OHom_{\op Q}(\op M, \op N)} \OEnd_{\op Q}(\op N).
\]
for the pullback of the diagram in Proposition \ref{prop:functorial}. Concretely, elements of $\OEnd_{\op Q}(f)(r)$ are pairs $(u,v)$ of elements $u\in  \OEnd_{\op Q} (\op M)(r)$ and $v\in  \OEnd_{\op Q} (\op N)(r)$ such that $f\circ u = v\circ f^{\otimes r}$.
We note that $\OEnd_{\op Q}$ naturally inherits the operad structure from $\OEnd_{\op Q} (\op M)\times \OEnd_{\op Q}(\op N)$.

\begin{proof}[Proof of Proposition \ref{prop:functorial}]
For a weak equivalence $f$ as in the Proposition that is either a cofibration or a fibration we define the zigzag of dg operads $F$ to be
\[
F : \OEnd_{\op Q} (\op M) \leftarrow \OEnd_{\op Q}(f) \rightarrow \OEnd_{\op Q} (\op N).
\]
We shall check that the morphisms above are weak equivalences.
Note that the diagram \eqref{equ:functorial pullback} becomes a pullback square.
But since one of the morphisms is a fibration in this case by the preceding Lemma, all morphisms in the diagram are weak equivalences by right properness of the category of dg symmetric sequences.

We shall also note that the assignment of $F$ is uniquely determined, up to homotopy, by the condition (iii) of the Proposition. Indeed, let 
\[
\OEnd_{\op Q} (\op M) \xleftarrow{p_1} X \xrightarrow{p_2} \OEnd_{\op Q} (\op N)
\]
be some other zigzag of operad quasi-isomorphisms making \eqref{equ:functorial pullback} commute. Then, from the universal property of the pullback we obtain the map of (a priori) symmetric sequences (the dashed arrow) fitting into a diagram
\[
\begin{tikzcd}
\OEnd_{\op Q} (\op M)  & \OEnd_{\op Q}(f) \ar{d} \ar{l} \\
X \ar[dashed]{ur}{(p_1,p_2)} \ar{r} \ar{u} & \OEnd_{\op Q} (\op N)
\end{tikzcd}\, .
\]
But since $p_1$ and $p_2$ are operad maps this is an operad map as well. It is also a quasi-isomorphism since all other morphisms are. Hence our two operad maps up to homotopy are homotopic, since they are the rims of a commutative diagram of weak equivalences.

Next we consider for $f$ a general weak equivalence, not necessarily a (co)fibration.
In this case we factorize $f= p\circ i$ into an acyclic cofibration $i$ followed by an acyclic fibration $p$, and just chain the zigzags obtained:
\[
\begin{tikzcd}
& & X \ar{dr}\ar{dl} & & \\
& \OEnd_{\op Q} (i) \ar{dr}\ar{dl}& & \OEnd_{\op Q} (p) \ar{dr}\ar{dl}& \\
\OEnd_{\op Q} (\op M) & & \OEnd_{\op Q} (\bullet) & &  \OEnd_{\op Q} (\op N).
\end{tikzcd}
\]
The element $X$ here may be taken such that the central square is a homotopy pullback of dg operads. One easily verifies that assertion (3) of the Proposition is then satisfied.

However, we need to verify that our zigzag is independent of the factorization $f= p\circ i$ chosen. To do this, we note that any two such factorizations are the outer rims of a diagram of acyclic (co)fibrations of the form 
\[
\begin{tikzcd}
 & \bullet \ar[twoheadrightarrow]{dr}  \ar[hookrightarrow]{d}& \\
\OEnd_{\op Q} (\op M) \ar[hookrightarrow]{r}\ar[hookrightarrow]{dr}\ar[hookrightarrow]{ur}& \bullet   \ar[twoheadrightarrow]{r} \ar[twoheadrightarrow]{d} & \OEnd_{\op Q} (\op N) \\
 & \bullet  \ar[twoheadrightarrow]{ur}& 
\end{tikzcd}\, .
\]
by standard lifting arguments. To check that the zigzags assigned to outer rim homotopy commutes it suffices to check that those assigned to the four triangles homotopy commute. But in each such, the composite is a fibration or cofibration, and hence the uniqueness statement shown just above implies that they homotopy commute.

Finally, we need to check compatibility with compositions of morphisms, i.e., assertion (2) of the proposition.
Since we may factor all morphisms into fibrations and cofibrations it suffices to check the following. Let $f=p\circ i$ be our given morphism. Then we need to check that the zigzag assigned to $f\circ p'$ or $f\circ i'$ or $p'\circ f$ or $i' \circ f$ is homotopic to the compositions of zigzags, where $p'$ ($i'$) are some other (co)fibrations.
The cases "$p'\circ f$" and "$f\circ i'$" are already covered by our uniqueness statement above since $i\circ i'$ is again an acyclic cofibration and $p'\circ p$ is an acyclic fibration.
Next consider $f\circ p'$ (resp. $\iota'\circ p$), which we factorize into an acyclic fibration and an acyclic cofibration again, say $\tilde p \circ \tilde i$. By lifting, our morphisms then can be fitted into diagrams of weak equivalences between cofibrant objects
\[
\begin{tikzcd}
\bullet \ar[hookrightarrow]{r}{i} & \bullet \ar[hookrightarrow]{dr}{p}&   \\
\OEnd_{\op Q} (\op M) \ar[hookrightarrow]{r}\ar[hookrightarrow]{dr}[swap]{\tilde i}\ar[twoheadrightarrow]{u}{p'} & \bullet \ar[twoheadrightarrow]{r} \ar[twoheadrightarrow]{u} & \OEnd_{\op Q} (\op N) \\
& \bullet  \ar[twoheadrightarrow]{ur}[swap]{\tilde p}  \ar[hookrightarrow]{u}& 
\end{tikzcd}
\]
and
\[
\begin{tikzcd}
& \bullet \ar[twoheadrightarrow]{r}{p}  \ar[hookrightarrow]{d}& \bullet \ar[hookrightarrow]{d}{i'} \\
\OEnd_{\op Q} (\op M) \ar[hookrightarrow]{r}\ar[hookrightarrow]{dr}[swap]{\tilde i}\ar[hookrightarrow]{ur}{i} & \bullet \ar[twoheadrightarrow]{r} \ar[twoheadrightarrow]{d} & \OEnd_{\op Q} (\op N) \\
& \bullet  \ar[twoheadrightarrow]{ur}[swap]{\tilde p}& 
\end{tikzcd}\, .
\]
The homotopy commutativity of the diagrams of operads associated to the triangles above is shown as before.
This leaves the squares. We can extend such a square by picking a one-sided inverse of one of the fibrations as follows
\[
\begin{tikzcd}
\bullet  \ar[twoheadrightarrow]{r} \ar[hookrightarrow]{d} & \bullet  \ar[hookrightarrow]{d} \ar[-]{ddr}{=}& \\
\bullet  \ar[twoheadrightarrow]{r} \ar[twoheadrightarrow]{drr} & \bullet \ar[twoheadrightarrow]{dr} & \\
 & & \bullet 
\end{tikzcd}
\]
Note also that the one-sided inverse is necessarily surjective, and hence a fibration in our model category. Now we equivalently check that the diagram associated to the outer rim (-triangle) and those associated to the inner triangles homotopy commute. But this is again covered as above.
\end{proof}

\begin{cor}
	The assignment of Proposition \ref{prop:functorial} from weak equivalences between cofibrant $\op Q$-modules to zigzags between the endomorphism operads sends homotopic morphisms to homotopic zigzags.
\end{cor}
\begin{proof}
	One definition of being homotopic for two weak equivalences is that they are the outer rim of a commutative diagram of weak equivalences.
	But by assertion (2) of the proposition it follows that such a diagram is sent to a big diagram of zigzags of operads all of whose faces homotopy commute, and hence the two outer zigzags define homotopic operad morphisms.
\end{proof}

\subsubsection{More explicit construction}\label{sec:funcdirectmap}
Let again $f: \op M \to \op N$ be a weak equivalence between cofibrant $\op Q$-modules.
Under favorable conditions we may in fact define a direct operad quasi-isomorphism between
$\OEnd_{\op Q}(\op M)$ and $\OEnd_{\op Q}(\op N)$. Concretely, suppose $f$ is either a fibration or a cofibration, say a cofibration to start with.
Suppose further that $\op M(0)=\op \op N(0)=\op Q(0)=0$ (as we always assume) and additionally $\op M(1)=\op M(2)=\K$.
Then the operads $\OEnd_{\op Q}(\op M)$ and $\OEnd_{\op Q}(\op N)$ carry a natural augmentation. In particular, a unital operad map $F:\OEnd_{\op Q}(\op M)\to \OEnd_{\op Q}(\op N)$ is defined by specifying a non-unital operad map 
\[
\bar F : \overline{\OEnd}_{\op Q}(\op M)\to \OEnd_{\op Q}(\op N),
\]
where $\overline{\OEnd}_{\op Q}(\op M)\subset \OEnd_{\op Q}(\op M)$ is the augmentation ideal.
We pick a one-sided inverse $g$ to $f$.
Then $\bar F$ can be defined on $u\in \overline{\OEnd}_{\op Q}(\op M)(r)$ as
\[
\bar F(u) = f\circ u \circ g^{\otimes r}.
\]

We leave it to the reader to verify that the honest operad map thus defined is indeed homotopic to the one defined via a zigzag above, and also to write down the dual construction in case $f$ is a fibration instead of a cofibration. (In the latter case one obtains a map in the reverse direction $\OEnd_{\op Q}(\op M)\leftarrow \OEnd_{\op Q}(\op N)$.)

\subsection{Torsors and operad morphisms}

Let now $\op P$ and $\op Q$ be operads in dg vector spaces.
An operadic $\op P$-$\op Q$-bimodule is a bimodule for the monoids $\op P$ and $\op Q$. Concretely, it is a
symmetric sequence $\op M$ together with a map 
\[
\op P \circ \op M\circ \op Q \to \op M
\]
satisfying some associativity and unit axioms. 
The category of operadic $\op P$-$\op Q$-bimodules also carries a (semi-)model structure constructed in \cite{Frbookmod}, by transfer from the standard model structure of (collections in) cochain complexes.

In this section, we will use mostly operadic bimodules of the following type.
\begin{defi}\label{def:torsor}
An operadic right $\op P$-$\op Q$-torsor is an operadic $\op P$-$\op Q$-bimodule $\op M$ satisfying the following conditions.
\begin{itemize}
\item $H(\op M(1))\cong \K$ is one dimensional. We denote some cocycle generating the nontrivial cohomology class by $\bo \in \op M(1)$.
\item The map of right $\op Q$-modules $\op Q\to \op M$ given by the composition with $\bo$, i.e., $q\mapsto \bo \circ q$, is a quasi-isomorphism. 
\end{itemize}
\end{defi}

For simplicity, we will furthermore assume that the operads and bimodules appearing here are reduced, in that $\op P(0)=\op Q(0)=\op M(0)=0$, while $\op P(1)=\op Q(1)=\op M(1)=\K$.
In the reduced case, the data $\bo$ in the above definition is determined uniquely up to multiplication by an unimportant nonzero scalar.
Furthermore, in this setting all operads come equipped with a canonical augmentation.

Now the key point of this Appendix is that any operadic right $\op P$-$\op Q$-torsor $\op M$ encodes an operad map up to homotopy 
\[
\op P \dashrightarrow \op Q.
\]
The construction is as follows. We assume that $\op M$ is cofibrant as a right $\op Q$-module.
If not, we may replace $\op M$ by some cofibrant resolution, respecting the left $\op P$-action.
(For example, one may take the usual bar-cobar resolution as right $\op Q$-module.) 

Next, since we required $\op M(1)=\K$ there is a canonical choice, up to a scalar, of a right $\op Q$-module quasi-isomorphism 
\[
\iota : \op Q \to \op M,
\]
sending the unit to some non-zero element of $\op M(1)$.
Our desired morphism $\op P\dashrightarrow \op Q$, is then encoded by the zigzag of dg operad morphisms
\begin{equation}\label{equ:bimod zigzag def}
\op P \to \OEnd_{\op Q}(\op M) \xleftarrow{\sim} \bullet \xrightarrow{\sim} \OEnd_{\op Q}(\op Q) = \op Q.
\end{equation}
Here the left-hand map is the one naturally induced by the left $\op P$-action on $\op M$.
The right-hand zigzag is the one associated to the morphism $\iota$ according to Proposition \ref{prop:functorial}. Concretely, we can take for the intermediate object $\OEnd(\iota)$. Usually one can also replace the zigzag by a direct operad morphism according to subsection \ref{sec:funcdirectmap}.

\begin{lemma}\label{lem:indep choices}
	The above construction is independent of choices, up to homotopy.
\end{lemma}
\begin{proof}
	Two choices have been made. First $\iota$ is only canonical up to scale. We leave it to the reader to check that this scale does not affect the morphism.
	The second choice is that of a cofibrant resolution of $\op M$ as right $\op Q$-module.
	We claim that any two such give rise to homotopic morphisms.
	It is sufficient to compare an arbitrary resolution to a fixed one, say we take the bar-cobar resolution $\widehat {\op M}$ of $\op M$ as a $\op P$-$\op Q$-bimodule.
	Then this is cofibrant also as $\op P$-$\op Q$-bimodule, and by lifting can assume that our other resolution, say $\op M'$, fits into a diagram of bimodules and actions as follows.
	\begin{equation}\label{equ:welldefineddg}
	 \begin{tikzcd}[column sep=0.5em]				
	\op P \arrow[-]{d}{=} &\aol & \widehat{\op M} \arrow{d}{\sim}[swap]{f} & \aor & \op Q \arrow[-]{d}{=} \\
	\op P & \aol  & \op M' &\aor & \op Q
	\end{tikzcd}
	\end{equation}
	We desire to show that the operad morphisms $\op P\dashrightarrow \op Q$ associated to the upper and lower row are homotopic.
	We first look at the left-hand parts of these morphisms (the maps from $\op P$ into the endomorphism operads of $\widehat{\op M}$, $\op M'$, see \eqref{equ:bimod zigzag def}).
	We also want to use the zigzag between $\OEnd_{\op Q}(\widehat{\op M})$ and $\OEnd_{\op Q}(\op M')$ we get from Proposition \ref{prop:functorial}. Suppose first that $f$ is either a cofibration or a fibration in the category of $\op P$-$\op Q$ bimodules. If $f$ is a fibration then it is a fibration in the category of right $\op Q$-modules as well, since the fibrations are just the surjective maps in either case. If $f$ is a cofibration, then it is a cofibration in right $\op Q$-modules as well, using \cite[Proposition 12.3.2]{Frbookmod}. In either case we have a commutative diagram
	\[
	\begin{tikzcd}
	& \OEnd_{\op Q}(\widehat{\op M}) \ar{dr}&  \\
	\op P \ar{r} \ar{ur}\ar{dr} & \OEnd_{\op Q}(f) \ar{u}{\sim}\ar{d}{\sim} & \OHom_{\op Q}((\widehat{\op M},\op M')  \\
	& \OEnd_{\op Q}(\op M') \ar{ur}& 
	\end{tikzcd}\, .
	\]
	The middle horizontal arrow comes from the universal property of the pullback $\OEnd_{\op Q}(f)$, using that the outer rim of the diagram commutes because $f$ is a map of $\op P$-$\op Q$-bimodules.
	(Note also that the two right-hand morphisms are only morphisms of symmetric sequences, while the other morphisms are morphisms of dg operads.)
	In the case that $f$ is neither a fibration nor a cofibration we factorize it into an acyclic cofibration followed by an acyclic fibration, and apply the above argument so each morphism in turn.
	
	Next we study the right-hand part of our maps $\op P\dashrightarrow \op Q$.
	It comes from the $\op Q$-module maps
	\[
	\begin{tikzcd}
	\widehat{\op M} \ar{d}{\sim}[swap]{f} & \op Q \ar[-]{d}{=} \ar{l}{\hat\iota}\\
	\op M' & \op Q \ar{l}{\iota'}
	\end{tikzcd},
	\]
	where we may pick $\hat\iota$ and $\iota'$ (as above) such that the diagram commutes.
	But then Proposition \ref{prop:functorial} implies that the associated diagram of operad quasi-isomorphisms 
	\[
		\begin{tikzcd}
	 \OEnd_{\op Q}(\widehat{\op M}) & \OEnd_{\op Q}(\hat \iota)  \ar{l}\ar{dr} &  \\
	 \bullet \ar{u}\ar{d} & &\OEnd_{\op Q}(\op Q)=\op Q  \\
	\OEnd_{\op Q}(\op M') &  \OEnd_{\op Q}(\iota) \ar{l}\ar{ur}& 
	\end{tikzcd}
	\]
	homotopy commutes. Hence, taking the left-hand and right-hand parts together, we have shown the Lemma.
\end{proof}

The next important point to check here is that the assignment from operadic right torsors to homotopy classes of maps sends quasi-isomorphic triples
\[
\op P \aol \op M \aor \op Q
\]
to quasi-isomorphic maps.

\begin{prop}\label{prop:mapfunctorial}
Suppose we are given a quasi-isomorphism of operadic $\op P$-$\op Q$ right torsor $\op M$ and an operadic $\op P$'-$\op Q'$ right torsor $\op M'$, i.e., a commuting diagram of maps and actions
\beq{equ:Propdiag}
 \begin{tikzcd}[column sep=0.5em]				
\op P \arrow{d}{\sim} &\aol & \op M \arrow{d}{\sim} & \aor & \op Q \arrow{d}{\sim} \\
 \op P' & \aol  & \op M' &\aor & \op Q'
 \end{tikzcd}\, .
\eeq
Then the induced maps $\op P \dashrightarrow \op Q$ and $\op P' \dashrightarrow \op Q'$ induced by the above construction from the bimodules $\op M$ and $\op M'$ are quasi-isomorphic.
\end{prop}

\begin{proof}
For simplicity, we may reduce the problem to the case when two of the three vertical arrows are the identity, by "extruding" the diagram as follows.
\[
 \begin{tikzcd}[column sep=0.5em]				
 \op P \arrow{d}{=} &\aol & \op M \arrow{d}{\sim} & \aor & \op Q \arrow{d}{=} \\
 \op P \arrow{d}{\sim} &\aol & \op M' \arrow{d}{=} & \aor & \op Q \arrow{d}{=} \\
 \op P' \arrow{d}{=} &\aol & \op M' \arrow{d}{=} & \aor & \op Q \arrow{d}{\sim} \\
 \op P' & \aol  & \op M' &\aor & \op Q'
 \end{tikzcd}
\]
Here we have used the left and right vertical maps of \eqref{equ:Propdiag} to define the left and right $\op P$ and $\op Q$-module structures on $\op M'$.

To begin with the easiest case, the functoriality in $\op P$ is obvious, i.e., if the middle and right vertical arrows of \eqref{equ:Propdiag} are identities then we clearly have a commutative diagram of operad maps
\[
\begin{tikzcd}
\op P \ar{dd} \ar{dr} & & \\
& \OEnd_{\op Q}(\op M) & \ar[swap]{l}{\sim} \bullet \ar{r}{\sim} & \OEnd_{\op Q}(\op Q)= \op Q\\
\op P' \ar{ur} & &
\end{tikzcd}\, .
\]
Secondly, assume that both the left and right vertical arrows of \eqref{equ:Propdiag} are identities. Then the proof of Lemma \ref{lem:indep choices} (cf. diagram \eqref{equ:welldefineddg}) shows that the upper and lower row yield homotopic morphisms.

Finally we check functoriality in $\op Q$, i.e., we assume that the left and middle vertical maps in \eqref{equ:Propdiag} are identities, and we are hence in the situation
\[
 \begin{tikzcd}[column sep=0.5em]				
\op P' \arrow{d}{=} &\aol & f^*\op M' \arrow{d}{\cong} & \aor & \op Q \arrow{d}[swap]{\sim}{f} \\
\op P' & \aol  & \op M' &\aor & \op Q'
\end{tikzcd}\, .
\]
To be clear we have made the use of the restriction functor $f^*$ explicit, which creates the $\op Q$-module $f^*\op M'$ from the $\op Q'$-module $\op Q'$.
We desire to show that the two operad morphisms $\op P'\dashrightarrow \op Q$ and $\op P\dashrightarrow \op Q'$ associated to the upper and lower row are quasi-isomorphic.
To this end we choose a cofibrant replacement $\op N$ of $f^*\op M'$ as right $\op Q$-module, 
say for concreteness by the bar-cobar construction
\[
\op N = \Barc_{\op Q}\Bar_{\op Q}(f^*\op M').
\]
We also choose a weak equivalence of right $\op Q$-modules $\iota : \op Q\to \op M$.
Then we apply the induction functor $\Ind_{\op Q}^{\op Q'}$ to obtain a map of $\op Q'$ modules (with $\iota'=\Ind_{\op Q}^{\op Q'} \iota$)
\[
\Ind_{\op Q}^{\op Q'} \op N \xleftarrow{\iota'} \Ind_{\op Q}^{\op Q'} \op Q = \op Q'.
\]
We note that these $\op Q'$-modules are also cofibrant since $\Ind_{\op Q}^{\op Q'}$ is left Quillen.
Furthermore, $\Ind_{\op Q}^{\op Q'}$ is a (dg) functor, and respects the presence of the left $\op P'$-action. Furthermore, again by functoriality and the fact that induction commutes with tensor products, it induces a map between the endomorphism operads.
Hence we obtain a commutative diagram of dg operad morphisms
\[
\begin{tikzcd}
\op P' \ar[-]{d}{=} \ar{r}& \OEnd_{\op Q}(\op N) \ar{d}& \OEnd_{\op Q}(\iota) \ar{r}{\sim}\ar{l}[swap]{\sim}\ar{d}& \OEnd_{\op Q} (\op Q)=\op Q \ar{d}{f}[swap]{\sim}\\
 \op P' \ar{r} & \OEnd_{\op Q}(\Ind_{\op Q}^{\op Q'}\op N) & \OEnd_{\op Q}(\iota') \ar{r}{\sim}\ar{l}[swap]{\sim}& \OEnd_{\op Q'} (\op Q')=\op Q'
\end{tikzcd}\, .
\]
Since $f$ is a quasi-isomorphism it follows that all vertical maps are as well.
Hence the upper and lower row form quasi-isomorphic operad morphisms.
The lower row is induced from the upper row of the diagram of bimodules
\[
\begin{tikzcd}[column sep=0.5em]				
\op P' \arrow[-]{d}{=}&\aol & \Ind_{\op Q}^{\op Q'}(\Barc_{\op Q}\Bar_{\op Q}(f^* \op M')) \arrow{d}{\simeq}  & \aor & \op Q' \arrow[-]{d}{=} \\
\op P' &\aol & \Barc_{\op Q'}\Bar_{\op Q'} \op M'   & \aor & \op Q' \\
\end{tikzcd}\, .
\]
The middle vertical arrow is obtained in the following two steps. First one notes that the bar-cobar construction is compatible with resrictions in the sense that there is a natural map of $\op P'$-$\op Q$-bimodules
\[
\Barc_{\op Q}\Bar_{\op Q}(f^* \op M') \to f^*(\Barc_{\op Q'}\Bar_{\op Q'}(f^* \op M')).
\]
Then one applies the counit of the induction-restriction adjunction to obtain the stated morphism. Invoking again the proof of Lemma \ref{lem:indep choices} we see that the upper and lower row here induce homotopic morphisms and hence we are done.

\end{proof}

\subsection{The operad map encoded by $\ESC_f$}
Suppose that $f: \op S\to \op T$ is a map of reduced topological $\La$-operads.
We associate to this map a two colored operad $\ESC_f$ as before.
We consider the operads of chains $\op P= C(\op S)$, $\op Q=C(\op T)$.
Furthermore, we consider the operadic right $\op P$-$\op Q$-torsor $\op M=C(\ESC_f(\bullet,0))$.
Then, more or less trivially, one observes the following result.
\begin{prop}
In the aforementioned setting, the operad map $\op P\to \op Q$ encoded by $\op M$ by the construction of the previous subsection is homotopic to the map induced by $f$. 
\end{prop}
\begin{proof}
In this case $\op M = \op Q$, hence it is already cofibrant as right $\op Q$ module, and our zigzag \eqref{equ:bimod zigzag def} becomes very simple:
\[
\begin{tikzcd}
\op P \ar{r} & \OEnd_{\op Q}(\op M) = \OEnd_{\op Q}(\op Q)= \op Q 
\end{tikzcd}\, .
\] 
Concretely, the map sends and $p\in \op P(r)$ to the left action $p\cdot (\bo,\dots,\bo)\in \op Q(r)$ on operadic units of $\op Q$. But if we compare to the definition of that left action in the construction of $\ESC_f$ in section \ref{sec:ESCfdef}, then we see that that left action is just composition with the image under $f$. Hence it is clear that our map is the same is (the chains map induced by) $f$.
\end{proof}

\bibliographystyle{plain}

\end{document}